\documentclass[12pt]{article}
\usepackage{a4wide, amsfonts, mathrsfs, enumerate, url, amsmath, amsthm}
\newtheorem{theorem}{Theorem}[section]
\newtheorem{cor}[theorem]{Corollary}
\newtheorem{prop}[theorem]{Proposition}
\newtheorem{lemma}[theorem]{Lemma}

\newcommand{\J}{\mathscr{J}}
\newcommand{\N}{\mathbb{N}}
\newcommand{\R}{\mathscr{R}}
\renewcommand{\L}{\mathscr{L}}
\newcommand{\PSL}{\mathop{\mathrm{PSL}}\nolimits}
\newcommand{\GLS}{\mathop{\mathrm{GLS}}\nolimits}
\newcommand{\GL}{\mathop{\mathrm{GL}}\nolimits}
\newcommand{\GF}{\mathop{\mathrm{GF}}\nolimits}
\newcommand{\gauss}[3]{{#1\atopwithdelims[]#2}_#3}
\newcommand{\ee}{\mathrm{e}}
\newcommand{\set}[2]{\{#1:#2\}}  

\begin{document}

\title{Chains of subsemigroups}
\author{Peter J. Cameron\\
\small{Mathematics Institute, University of St Andrews,}\\
\small{North Haugh, St Andrews KY16 9SS, UK}\\
Maximilien Gadouleau\\
\small{School of Engineering and Computing Sciences, University of Durham,}\\
\small{Lower Mountjoy South Road, Durham DH1 3LE, UK}\\
James D. Mitchell\\
\small{Mathematics Institute, University of St Andrews,}\\
\small{North Haugh, St Andrews KY16 9SS, UK}\\
Yann Peresse\\
\small{School of Physics, Astronomy and Mathematics, University of
Hertfordshire}\\
\small{Hatfield, Herts AL10 9AB, UK}}
\date{Draft 11, January 2015}
\maketitle

\begin{abstract}
  We investigate the maximum length of a chain of subsemigroups in various
  classes of semigroups, such as the full transformation semigroups, the general
  linear semigroups, and the semigroups of order-preserving transformations of
  finite chains. In some cases, we give lower bounds for the total number of
  subsemigroups of these semigroups. We give general results for finite
  completely regular and finite inverse semigroups. Wherever possible, we state
  our results in the greatest generality; in particular, we include infinite
  semigroups where the result is true for these.

  The length of a subgroup chain in a group is bounded by the logarithm of the
  group order. This fails for semigroups, but it is perhaps surprising that
  there is a lower bound for the length of a subsemigroup chain in the full
  transformation semigroup which is a constant multiple of the semigroup order.
\end{abstract}

\section{The definition}

Let $S$ be a semigroup.  A collection of subsemigroups of $S$ is
called a \emph{chain} if it is totally ordered with respect to inclusion.  In
this paper we consider the problem of finding the longest chain of subsemigroups
in a given semigroup. From among several conflicting candidates for the
definition, we define the \textit{length} of a semigroup $S$ to be the largest
number of non-empty subsemigroups of $S$ in a chain minus 1; this is denoted
$l(S)$. There are several reasons for choosing this definition rather than
another, principally: several of the formulae and proofs we will present are
more straightforward with this definition (especially that in
Proposition~\ref{prop-ideal}, which is the basis for several of our results);
when applied to a group our definition of length coincides with the definition
in the literature (for more details see Section~\ref{section-groups}).  There
are some negative aspects of this definition too.  For example, if $S$ is a null 
semigroup (the product of every pair of elements equals $0$), then $l(S) = |S| -
1$; or if $S$ is empty, then $l(S) = - 1$. Our definition of length also
differs from the usual order-theoretic definition. 

The paper is organised as follows: we review some known results for groups in
Section~\ref{section-groups}; in Section~\ref{section-generalities} we present
some general results about the length of a semigroup and its relationship to the
lengths of its ideals; in Section~\ref{section-full-trans} we give a lower bound
for the length of the full transformation semigroup on a finite set, and
consider the asymptotic behaviour of this bound; in
Sections~\ref{section-order-preserving} and~\ref{section-general-linear} we
perform an analysis similar to that for the full transformation semigroup for
the semigroup of all order-preserving transformations, and for the
general linear monoid; in Sections~\ref{section-inverse}
and~\ref{section-completely-regular} we provide a formula for the length of an
arbitrary finite inverse or completely regular semigroup in terms of the lengths
of its maximal subgroups, and its numbers of $\L$- and $\R$-classes. In
Section~\ref{section-numbers}, as consequences of our results about the full
transformation monoid, we give some bounds on the number of subsemigroups, and
the maximum rank of a subsemigroup, of the full transformation monoid. 

Note that, with the exception of Proposition~\ref{prop-ideal}, all semigroups
we consider are finite.

\section{Subgroup chains in groups}\label{section-groups}

In this section we give a brief survey of a well-understood case, that of
groups. We will use some of the results in this section later in the paper.

The question of the length $l(G)$ of the longest chain of subgroups in a
finite group has been considered for some time. The base and strong generating
set algorithm for finite permutation groups, due to Charles Sims, involves
constructing a chain of point stabilisers. L\'aszl\'o Babai~\cite{babai}
pointed out that the length of such a chain in any action of $G$ is bounded
above by $l(G)$, so this parameter is important in the complexity analysis
of the algorithm. Babai gave a bound, linear in $n$, for the length of the
symmetric group $S_n$; the precise value of $l(S_n)$ was computed by Cameron,
Solomon and Turull~\cite{cst}: values are given in sequence A007238 in the
On-Line Encyclopedia of Integer Sequences \cite{sloane0}.

\begin{theorem}\label{symmetric_theorem}
  $l(S_n)=\lceil 3n/2\rceil - b(n) - 1$, where
  $b(n)$ is the number of ones in the base~$2$ expansion of $n$.
\end{theorem}

It is easy to see that, if $N$ is a normal subgroup of $G$, then
$l(G)=l(N)+l(G/N)$. (In one direction, there is a chain of length $l(N)+l(G/N)$
passing through $N$; in the other direction, if $H$ and $K$ are subgroups of $G$
with $H<K$, then either $H\cap N<K \cap N$ or $HN/N<KN/N$, so in any subgroup
chain in $G$, each step involves taking a step in either $N$ or $G/N$.) So, for
any group $G$, we obtain $l(G)$ by summing the lengths of the composition
factors of $G$, and the problem is reduced to evaluating the lengths of the
finite simple groups. The result cited in the preceding paragraph deals with the
alternating groups. The problem was further considered by Seitz, Solomon and
Turull~\cite{st,gst}. It is not completely solved for all finite simple groups,
but we can say that it is reasonably well understood. In what follows, we will
regard a formula containing $l(G)$ for some group $G$ as ``known''.

We will use a special case of the following (known) result later. The function
$\Omega(n)$ gives the number of prime divisors of $n$, counted with their
multiplicities; equivalently, the number of prime power divisors of $n$.

\begin{prop}
  \label{p:sol}
  For any group $G$, $l(G)\le\Omega(|G|)$. Equality holds if and only if each
  non-abelian composition factors of $G$ is a $2$-transitive permutation group
  of prime degree in which the point stabiliser $H$ also satisfies
  $l(H)=\Omega(|H|)$. In particular, any soluble group $G$ satisfies
  $l(G)=\Omega(|G|)$.
\end{prop}

\paragraph{Remark} It follows from the Classification of Finite Simple Groups
that the non-abelian simple groups with this property are 
$\PSL(2,2^a)$ where $2^a+1$ is a Fermat prime,
$\PSL(2,7)$, $\PSL(2,11)$, $\PSL(3,3)$ and $\PSL(3,5)$.

\begin{proof}
It is clear from Lagrange's Theorem that $l(G)\le\Omega(|G|)$. Equality holds
if and only if it holds in every composition factor. 

A non-abelian finite simple group with this property has a subgroup of prime
index, and so acts as a transitive permutation group of prime degree. By
Burnside's theorem, it is $2$-transitive. The rest of the proposition is clear.
\end{proof}

Since a subsemigroup of a finite group is a subgroup, these results solve
particular cases of our general problem.

\section{Generalities about subsemigroup chains}\label{section-generalities}

In contrast to groups, where the length of a chain of subgroups is at most the
logarithm of the group order, a semigroup may have a chain whose length is equal
one less than its order. A null semigroup of any order has this property, as does
any semigroup which is not generated by a proper subset (i.e.\ any semigroup
whose $\J$-classes are semigroups of left or right zeros, and where $S/\J$ is a
chain).

If $S$ is a semigroup and $T$ is a subsemigroup of $S$, then $l(T) \leq l(S)$.
Similarly, if $\rho$ is a congruence on $S$, then, since subsemigroups are
preserved by homomorphisms,  $l(S/\rho) \leq l(S)$.

Let $I$ be an ideal of the semigroup $S$ and let $S/I$ denote the \textit{Rees
quotient} of $S$ by $I$, i.e.\ the semigroup with the elements of $S\setminus I$
together with an additional zero $0\not\in S$; the multiplication is given by
setting $xy=0$ if the product in $S$ lies in $I$, and letting it have its value
in $S$ otherwise. As noted above, in the following result we do not assume
any finiteness condition.

\begin{prop}[cf. Lemma 1 in \cite{Ganyushkin2011aa}]
  \label{prop-ideal}
    Let $S$ be a semigroup and let $I$ be an ideal of $S$. Then
    $l(S)=l(I)+l(S/I)$. 
\end{prop}

\begin{proof} 
  We start by showing that $l(S)\geq l(I)+l(S/I)$.  Suppose that
  $\set{U_{\alpha}}{\alpha\hbox{ an ordinal},\ \alpha<l(I)}$ and 
  $\set{V_{\alpha}}{\alpha\hbox{ an ordinal},\ \alpha<l(S/I)}$ are 
  chains of non-empty proper subsemigroups of $I$ and $S/I$, respectively. Then 
  \begin{equation*}
    W_{\alpha}=
    \begin{cases}
      U_{\alpha} &\text{if }\alpha<l(I)\\
      (V_{\beta}\setminus \{0\})\cup I & \text{if }\alpha=l(I)+\beta<l(I)+l(S/I)
    \end{cases}
  \end{equation*}
  is a chain of $l(I)+l(S/I)$ proper subsemigroups of $S$, and so $l(S)\geq
  l(I)+l(S/I)$. 

  Suppose that $\mathcal{C}=\set{U_{\alpha}}{\alpha\hbox{ an ordinal},\
  \alpha<l(S)}$ is a chain of non-empty proper subsemigroups such that
  $U_{\alpha}<U_{\alpha+1}$ for all $\alpha<l(S)$.  We will show that we may
  assume, without loss of generality, that for all  $\alpha<l(S)$ either:
  \begin{equation}\label{dichotomy}
    \big(U_{\alpha+1}\setminus U_{\alpha}\big) \cap I=\emptyset\qquad\hbox{ or
    }\qquad U_{\alpha+1}\setminus U_{\alpha}\subseteq I .
  \end{equation}
  Since the union of a subsemigroup and an ideal is a semigroup, it follows that
  $U_\alpha\cup (U_{\alpha+1}\cap I)$ is a subsemigroup of $S$. Hence 
  $$U_{\alpha}\leq U_\alpha\cup (U_{\alpha+1}\cap I)\leq U_{\alpha+1}.$$
  If $l(S)$ is finite, then either $U_\alpha\cup (U_{\alpha+1}\cap
  I)=U_{\alpha}$ or $U_\alpha\cup (U_{\alpha+1}\cap I)=U_{\alpha+1}$. Therefore
  $\big(U_{\alpha+1}\setminus U_{\alpha}\big) \cap I=\emptyset$ or
  $U_{\alpha+1}\setminus U_{\alpha}\subseteq I $, respectively.

  Suppose that $l(S)$ is infinite. Then  replacing any subchain
  $U_{\alpha}<U_{\alpha+1}$  of $\mathcal{C}$ which fails (\ref{dichotomy}) by
  $U_{\alpha}<U_\alpha\cup (U_{\alpha+1}\cap I)<U_{\alpha+1}$ we obtain another
  chain of length $l(S)$. Furthermore, $U_\alpha\cup (U_{\alpha+1}\cap
  I)<U_{\alpha+1}$ and $U_{\alpha}<U_\alpha\cup (U_{\alpha+1}\cap I)$ satisfy
  (\ref{dichotomy}). 

  Assume without loss of generality that $\mathcal{C}$ satisfies
  (\ref{dichotomy}).  Note that $\set{U_\alpha\cap I}{\alpha<l(S)}$ is a chain
  of non-empty subsemigroups of $I$ and $\set{U_{\alpha}/I}{\alpha<l(S)}$
  is a chain of non-empty proper subsemigroups of $S/I$.  By (\ref{dichotomy}),
  for all $\alpha<l(S)$ either 
  $$U_{\alpha}\cap I=U_{\alpha+1}\cap I\quad\hbox{ and }\quad
  U_{\alpha}/I<U_{\alpha+1}/I$$
  or 
  $$U_{\alpha}/I=U_{\alpha+1}/I\quad\hbox{ and }\quad U_{\alpha}\cap
  I<U_{\alpha+1}\cap I$$
  Therefore $l(S)\leq l(I)+l(S/I)$.
\end{proof}

If $S$ is a semigroup and $x,y\in S$, then we write $x\J y$ if the principal
(two-sided) ideal $S^1xS^1$ generated by $x$ equals the ideal $S^1yS^1$
generated by $y$. The relation $\J$ is an equivalence relation called
\emph{Green's $\J$-relation}, and the equivalence classes are called
\emph{$\J$-classes}.  If $J_1$ and $J_2$ are $\J$-classes of a semigroup, then
we write $J_1\leq_{\J}J_2$ if $S^1xS^1\subseteq S^1yS^1$ for any $x\in J_1$ and
$y\in J_2$. It is straightforward to verify that $\leq_{\J}$ is a partial order
on the $\J$-classes of $S$. 

If $J$ is a $\J$-class of a finite semigroup $S$, then its \emph{principal
factor} $J^*$ is the semigroup with elements $J\cup \{0\}$ ($0\not\in J$) and
the product $xy$ of $x,y\in J$ defined to be its value in $S$ if $x,y, xy\in J$
and $0$ otherwise. In other words, if $J$ is not minimal, then $J^*$ is the Rees
quotient of the principal ideal generated by any element of $J$ by the ideal
consisting of those elements in $S$ whose $\J$-classes are not greater than $J$
under $\leq_{\J}$. If $J$ is minimal, then $J$ is a subsemigroup of $S$, and
$J^*$ is not isomorphic to the quotient in the previous sentence (which is
isomorphic to $J$), since $J^*$ has one more element. 

A semigroup $S$ is \emph{regular} if for every $x\in S$ there is $y\in S$ such
that $xyx=x$.

\begin{lemma} \label{lemma-regular}
  Let $S$ be a finite regular semigroup and let $J_1, J_2, \ldots, J_m$ be the
  $\J$-classes of $S$. Then $l(S)=l(J_1^*)+l(J_2^*)+\cdots +l(J_m^*)-1$.
\end{lemma}

\begin{proof}
  Assume without loss of generality that $J_1$ is maximal in the partial order
  of $\J$-classes on $S$. It follows that $I=S\setminus J_1$ is an ideal. Hence
  by Proposition~\ref{prop-ideal} it follows that $l(S)=l(I)+l(S/I)$.
 If $I=\emptyset$, then $S=J_1$, and so 
  $l(S)=l(J_1)=l(J_1^*)-1$, in which case we are finished.

 Suppose that $I\not=\emptyset$. Then $S/I$ is isomorphic to $J_1^*$ and so
  $l(S)=l(I)+l(J_1^*)$. Since $S$ is regular and $I$ is an ideal, it
  follows by Proposition A.1.16 in \cite{Rhodes2009aa} that $I$ is regular and
  the $\J$-classes of $I$ are $J_2, J_3, \ldots, J_m$.  Therefore repeating the
  argument in the previous paragraph a further $m-2$ times, we obtain
  $$l(S)=l(J_1^*)+l(J_2^*)+\cdots +l(J_m^*)-1,$$
  as required.
\end{proof}

We conclude this section with a simple application of the results in this, and
the previous, sections.  

\begin{prop}
  Let $S$ be a semigroup generated by a single element $s$ and let $m,n\in\N$ be
  the least numbers such that $s^{m+n}=s^{m}$.  Then $l(S)=m+\Omega(n)-1$,
  where $\Omega(n)$ is the number of prime power divisors of $n$. 
\end{prop}

\begin{proof}
  The $\J$-classes of $S$ are
  $$\{\{s\}, \{s^2\}, \ldots, \{s^{m-1}\}, \{s^m, s^{m+1},\ldots,
  s^{m+n-1}\}\},$$
  where the non-singleton class is the cyclic group
  $C_n$ with $n$ elements. 
  By repeatedly applying Proposition~\ref{prop-ideal}, 
  $$l(S)=m+l(C_n)-1,$$
  and $l(C_n)=\Omega(n)$ by Proposition~\ref{p:sol}.
\end{proof}

\section{The full transformation semigroup}\label{section-full-trans}

\subsection{Long chains}

The full transformation semigroup, denoted $T_n$, consists of all functions with
domain and codomain $\{1,\ldots,n\}$ under the usual composition of functions.
Clearly $|T_n|=n^n$. In this section, we will prove the following theorem.

\begin{theorem}\label{thm-full-transformation-monoid}
  $$l(T_n) \ge \ee^{-2} n^n - 2\ee^{-2}(1-\ee^{-1}) n^{n-1/3} - o(n^{n-1/3}).$$
\end{theorem}

The \emph{rank} of an element of $T_n$ is the cardinality of its image. The
$\mathscr{J}$-classes of $T_n$ are the sets $J_k$ of all elements of rank $k$.
Since $T_n$ is regular, Lemma~\ref{lemma-regular} implies that $l(T_n)$ is the
sum of the lengths of the principal factors $J_k^*$ of its
$\mathscr{J}$-classes, minus $1$.

A element $f$ of rank $k$ in $T_n$ has a kernel, which is the partition of
$\{1,\ldots,n\}$ into its pre-images (hence with $k$ parts), and an image, a
$k$-subset of $\{1,\ldots,n\}$. The set of all maps with given kernel $Q$ and
given image $A$ is an $\mathscr{H}$-class in the semigroup $T_n$, and has
cardinality $|A|!$.  So the number of maps of rank $k$ is
$$N(n,k)=S(n,k){n\choose k}k!,$$
where $S(n,k)$ is the Stirling number of the second kind.

If $f_1$ and $f_2$ are two maps of rank $k$ with kernels $Q_1,Q_2$ and images
$A_1,A_2$ respectively, then $f_1f_2$ has rank $k$ if $A_1$ is a transversal for
the partition $Q_2$, and smaller rank otherwise. So, if $P$ is a set of
$k$-partitions of $\{1,\ldots,n\}$ (partitions with $k$ parts), and $S$ a set of
$k$-subsets, with the property that no element of $S$ is a transversal for any
element of $Q$, then the set of maps with kernel in $P$ and image in $S$ is a
null semigroup in $J_{k}^*$.  We call a set $(P,S)$ with this property a
\emph{league}, and define its \emph{content} to be $|P|\cdot|S|$.

If a league $(P,S)$ of rank $k$ has content $m$, then the set of all maps $f$
with kernel in $P$ and image in $S$ has the property that the product of any two
of its elements has rank smaller than $k$; so this set, together with zero,
forms a null subsemigroup of the principal factor $J_k^*$ of order $1+m\cdot
k!$. This semigroup has a chain of subsemigroups of length equal to one less
than its order. Combining these observations with Lemma~\ref{lemma-regular}, we
obtain the following result.

\begin{prop}
  Let $F(n,k)$ be the largest content of a league of rank $k$ on
  $\{1,\ldots,n\}$. Then
  $$l(T_n)\ge\sum_{k=1}^nF(n,k)k! - 1.$$
\end{prop}

We prove Theorem~\ref{thm-full-transformation-monoid} by a suitable choice of
leagues, as follows. Choose one element of the set $\{1,\ldots,n\}$, say $n$;
let $P$ be the set of all $k$-partitions having $n$ as a singleton part, and $S$
the set of all $k$-subsets not containing $n$. Then clearly $(P,S)$ is a league,
and its content is 
$${n-1\choose k}S(n-1,k-1).$$

\begin{lemma}
  The expected rank $E(n)$ of a transformation in $T_n$ chosen uniformly at random
  satisfies
  $$
  E(n) = (1 - \ee^{-1}) n + O(1).
  $$
  Moreover, the standard deviation $\sigma(n)$ of the rank satisfies
  \begin{equation}
    \sigma(n) \le \sqrt{\ee^{-1} - 2 \ee^{-2}} \sqrt{n+1}
  \end{equation}
  for $n$ large enough.
\end{lemma}
\begin{proof}
  The exact values of the expectation $E(n)$ and of the variance $V(n)$ are given
  in \cite{hig1}, where their asymptotic estimates are also given. The expected
  rank is given by
  $$
  E(n) = n \left[ 1- \left( 1- \frac{1}{n} \right)^n \right] = (1- \ee^{-1})n + O(1).
  $$
  For the variance, we have
  \begin{eqnarray*}
    V(n) &=& n \left[ \left( 1- \frac{1}{n} \right)^n - \left( 1- \frac{2}{n}
    \right)^n \right] + n^2 \left[ \left( 1- \frac{2}{n} \right)^n - \left( 1-
    \frac{1}{n} \right)^{2n} \right]\\
    &=& n \left[\ee^{-1} \left(1 - \frac{1}{2n} + o(n^{-1}) \right) - \ee^{-2}
    \left( 1- \frac{2}{n}  + o(n^{-1}) \right) \right]\\
    &&{} + n^2 \left[ \ee^{-2} \left( 1 - \frac{2}{n} - \frac{2}{3n^2} +
    o(n^{-2}) \right) - \ee^{-2} \left( 1 - \frac{1}{n} - \frac{1}{6n^2} +
    o(n^{-2}) \right)\right]\\
    &=& n(\ee^{-1} - 2\ee^{-2}) + \frac{3 \ee^{-2} - \ee^{-1}}{2} + o(1).
  \end{eqnarray*}
  Since $\frac{3 \ee^{-2} - \ee^{-1}}{2} < \ee^{-1} - 2\ee^{-2}$, we have $V(n)
  \le (\ee^{-1} - 2\ee^{-2})(n+1)$ for $n$ large enough.
\end{proof}

We now return to the proof of Theorem~\ref{thm-full-transformation-monoid}.
Let $\tau = \sqrt{\ee^{-1} - 2 \ee^{-2}}$ and $K = \{k :
|k-E(n-1)| < n^{1/6} \tau n^{1/2}\}$; we then have
$$ n - k \ge n - E(n-1) - \tau n^{2/3} = \ee^{-1} n - \tau n^{2/3} - o(n^{2/3})$$
for any $k \in K$.  Also, for all $n$ large enough, $K$ contains all $k$ such
that $|k - E(n-1)| < n^{1/6} \sigma(n-1)$. Chebyshev's inequality then yields
$$ \sum_{k \in K} N(n-1,k-1) \ge (n-1)^{n-1} \left( 1 - n^{-1/3} \right) \ge
\ee^{-1} n^{n-1} (1- n^{-1/3}).
$$

Therefore, we obtain an overall chain of length at least
\begin{eqnarray*}
  \sum_{k \in K} {n-1 \choose k} S(n-1, k-1) k!
  &=&\sum_{k \in K} (n-k) N(n-1,k-1)\\
  &\ge& (\ee^{-1} n - \tau n^{2/3} - o(n^{2/3})) \ee^{-1} n^{n-1} \left( 1 - n^{-1/3} \right)\\
  &=& \ee^{-2} n^n - 2\ee^{-2}(1-\ee^{-1})(n^{n-1/3}) - o(n^{n-1/3}).
\end{eqnarray*}

\subsection{Combinatorial results}

The question of finding $F(n,k)$, the largest possible content of a league
$(P,S)$, where $P$ is a set of $k$-partitions and $S$ a set of $k$-subsets
of $\{1,\ldots,n\}$, is purely combinatorial, and maybe of some interest.
We give here some general bounds and some exact values.

We showed above that
\begin{equation}\label{equation-1}
  F(n,k)\ge{n-1\choose k}S(n-1,k-1).
\end{equation}
Another strategy gives a different bound, which is better for small $k$:
for $n\ge2$, we have
\begin{equation}\label{equation-2}
  F(n,k)\ge{n-2\choose k-2}S(n-1,k).
\end{equation}
This is proved by letting $S$ consist of all $k$-sets containing $1$ and $2$,
and $P$ the set of all $k$-partitions not separating $1$ and $2$.
Further improvements are possible.

In the extreme cases, we can evaluate $F(n,k)$ precisely, as follows. 

\begin{prop}
  \begin{enumerate}
    \item $F(n,1)=0$.
    \item For $n>3$, $F(n,2)=3(2^{n-3}-1)$, and a pair $(P,S)$ meets the bound
      if and only if $S$ is the set of edges of a triangle $T$ and $P$ is the set
      of $2$-partitions with $T$ contained in a part.
    \item $F(n,n-1)=s^2(2s-1)$, $s^2(2s+1)$, or $s(s+1)(2s+1)$ when
      $n=3s$, $3s+1$, or $3s+2$, respectively, with $s\ge1$.
    \item $F(n,n)=0$.
  \end{enumerate}
\end{prop}

\begin{proof}
  In the first and last case, the proof is trivial and ommited.

  \smallskip

  \textbf{(b)} Consider the case $k=2$. Then $S$ is the set of edges of a graph,
  and the partitions in $P$ do not cross edges, so each part of such a partition
  is a union of connected components of a graph. We are going to make moves which
  will all increase $|S|\cdot|P|$.

  First, by including edges so that each component is a complete graph, and by
  including all partitions whose parts are unions of components, we do not
  decrease $|S|\cdot|P|$. So we may assume that this is the case. Thus, if the
  components have sizes $a_1,\ldots,a_r$, then
  $$|S|\cdot|P| = \left(\sum_{i=1}^r{a_i\choose 2}\right)(2^{r-1}-1),$$
  where $\sum_{i=1}^ra_i=n$.

  Next we claim that, if $a_i\ge4$, then we can increase $|S|\cdot|P|$ by
  replacing the part of size $a_i$ by two parts of sizes $1$ and $a_i-1$. For we
  increase $r$ by $1$, so the second factor more than doubles; so it will suffice
  to show that $${a_i-1\choose2}\ge\frac{1}{2}{a_i\choose 2},$$ since then the
  first factor will be at least half of its previous value. Now the displayed
  inequality is equivalent to $a_i\ge4$.

  Also, splitting a part of size $2$ into two parts of size $1$ more than
  doubles the second factor and reduces the first factor by $1$. So this is
  also an improvement (except in the case where the resulting partition has
  all $a_i=1$, when the product is zero).

  So we can continue the process, increasing the objective function, until all
  $a_i$ are equal to $1$ or $3$. 

  If two $a_i$ are equal to $3$, then replacing them by $5$ and $1$ improves
  the sum, since
  $$2{3\choose2}<{5\choose 2}.$$
  Then we can replace the $5$ by three parts of sizes $3$, $1$ and $1$, by the
  preceding argument.

  So we end with a part of size $3$ and $n-3$ parts of size $1$, giving the
  value $3(2^{n-3}-1)$ claimed, and also the extremal configuration described.

  \smallskip

  \textbf{(c)} Now consider the case $k=n-1$. Let $(P,S)$ satisfy the conditions.
  Identify each element of $S$ by the single point it omits, and each element of
  $P$ by the pair of points (or edge) in the same class; then the condition
  asserts that no point of $P$ is on an edge of $S$. So to optimise we want $P$ to
  be a complete graph on, say, $m$ points, and $S$ to consist of the remaining
  $n-m$ points. Then $|S|\cdot|P|=m(m-1)(n-m)/2$. 

  This is maximised when $m$ is roughly $2n/3$; a detailed but elementary
  calculation gives the stated result.
\end{proof}

Table~\ref{t:leagues} gives some further exact values, computed with the
GAP~\cite{gap} package GRAPE~\cite{grape}, except the value of $F(7, 4)$, which
was computed by Chris Jefferson using the Minion \cite{minion} constraint
satisfaction solver.  Each table entry also gives a lower bound, which is the
maximum of the values in (\ref{equation-1}) and (\ref{equation-2}).  The column
headed ``Total'' multiplies $F(n,k)$ (or the lower bound) by $k!$, and sums over
$k$.  The entries in columns $k=1$ and $k=n$ are zero and have been omitted.

\begin{table}[ht]
  \begin{equation*}
    \begin{array}{||r||r||r|r|r|r|r||}
      \hline
      n & \hbox{Total} & k=2 & 3 & 4 & 5 & 6 \\
      \hline
      2 & 0,0 &&&&&\\
      3 & 2,2 & 1,1 &&&&\\
      4 & 24,18 & 3,3 & 3,2 &&&\\
      5 & 330,326 & 9,7 & 28,28 & 6,6 &&\\
      6 & 5382,5130 & 21,15 & 150,150 & 125,125 & 12,10 &\\
      7 & 98250,93782 & 45,31 &760,620 & 1350,1350 & 390,390 & 20,15 \\
      \hline
    \end{array}
  \end{equation*}
  \caption{\label{t:leagues}Values and bounds for $F(n,k)$}
\end{table}

A kind of dual problem, which is also connected to the theory of transformation
semigroups (though not to the questions considered here) is the following:
\begin{quote}
Given $n$ and $k$, what is the smallest size of a collection of $k$-subsets
of $\{1,\ldots,n\}$ which contains a transversal to every $k$-partition of
$\{1,\ldots,n\}$?
\end{quote}

For some asymptotic results about this question, see \cite{bt}; for an
application to semigroups, in the special case where there is a permutation
group $G$ such that every orbit of $G$ on $k$-sets has this property,
see \cite{ac}.

\section{Order-preserving transformations}\label{section-order-preserving}

A transformation $f\in T_n$ is \textit{order-preserving}  if $(i)f < (j)f$
whenever $i < j$.  In this section we consider $O_n$, the semigroup of all
order-preserving transformations of $\{1,\ldots,n\}$.  It is shown in
\cite{hig2}, for example, that $|O_n| = {2n-1 \choose n}$.

We will denote by $F^*(n,k)$ denote the maximum content of a league $(P, S)$
where $P$ consists of $k$-partitions corresponding to kernels of
order-preserving transformations, and $S$ is an $k$-element subset of
$\{1,2,\ldots, n\}$. 

\begin{prop}
$$l(O_n) \ge {2n - 3 \choose n} - 1 
         = \frac{(n-1)(n-2)}{(2n-1)(2n-2)} |O_n| - 1.$$
\end{prop}

Note that this lower bound is asymptotically $|O_n|/4$.

\begin{proof}
It is well known that each $\mathscr{H}$-class in $O_n$ is a singleton. For any
given value of the rank $k$, there are ${n \choose k}$ choices for the image of
a transformation in $O_n$, and ${n-1 \choose k-1}$ choices for its kernel, which
must be a partition of $\{1,\ldots,n\}$ into $k$ intervals \cite{gm} -- this is
because we specify such a partition by giving the $k-1$ points which divide the
interval appropriately. Therefore, the number of transformations of
rank~$k$ in $O_n$ is given by
$$ N^*(n,k) = {n \choose k} {n-1 \choose k-1}. $$

We can apply the same strategy as in $T_n$ in order to obtain long chains of
subsemigroups in $O_n$. Let $S$ be the set of all $k$-subsets of
$\{1,\ldots,n\}$ not containing $n$ and let $P$ be the set of partitions of
$\{1,\ldots,n\}$ into $k$ intervals such that the last interval is $\{n\}$. We
then have that $|S| = {n-1 \choose k}$ and $|P| = {n-2 \choose k-2}={n-2\choose
n-k}$ and so 
\begin{equation}\label{equation-3}
  F^*(n,k) \geq {n-1 \choose k} {n-2 \choose n-k}.
\end{equation}
Hence we have a chain of length
$$ \sum_{k=1}^n {n-1 \choose k} {n-2 \choose n-k} - 1 = {2n-3 \choose n} - 1, $$
using the Vandermonde convolution: ${m+n \choose k}=\sum_{i=0}^{k}{m \choose
i}{n \choose k-i}$. 
\end{proof}

As we did for the full transformation monoid, in the extreme cases, we can
evaluate $F^*(n,k)$ precisely, as follows. 

\begin{prop}
  \begin{enumerate}
    \item $F^*(n,1) = 0$.

    \item
      \begin{equation*}
        \begin{array}{lllr}
          F^*(n,2)=
          \max &\left\{\right.&\frac{1}{2} (n-\lfloor r^* \rfloor +1)(n- \lfloor
          r^* \rfloor)(\lfloor r^* \rfloor -1),\\
          && \frac{1}{2} (n-\lceil r^* \rceil +1)(n- \lceil r^*
          \rceil)(\lceil r^*
          \rceil -1) &\left.\right\}
        \end{array}
      \end{equation*}
      where $r^* = \left(2(n+1) - \sqrt{(n+1)^2 - 3n}\right)/3$.
    \item	
      $\displaystyle{F^*(n,n-1)=
      \left \lfloor \frac{n-1}{2} \right\rfloor \left \lceil
      \frac{n-1}{2} \right\rceil.}$

    \item $F^*(n,n) = 0$.
  \end{enumerate}
\end{prop}

The bound in (b) is asymptotically $(2/27)n^3$; that in (c), $n^2/4$.

\begin{proof}
  The proofs are very similar to the case of arbitrary leagues; as such, we shall
  use a similar notation.

  \smallskip

  \textbf{(b)} Again, we can represent $S$ as a graph and each part of any
  partition in $P$ is a union of connected components of that graph. We can still
  assume that $S$ forms a union of $r$ cliques, of cardinalities $a_1,\ldots,a_r$.
  However, for a graph with $r$ connected components, there are at most $r-1$
  possible choices for a partition in $P$, with equality if and only if the vertex
  set of each connected component is an interval. Therefore, the maximum content
  of a league with partitions into intervals is given by
  $$
  \max_{1 \le r \le n} \max_{a_1,\ldots,a_r} \left\{ (r-1) \sum_{i=1}^r
  \frac{a_i(a_i-1)}{2} \right\},
  $$
  where the inner maximum is taken over all $a_1,\ldots,a_r$ such that $a_i \ge 1$
  for all $i$ and $\sum_{i=1}^r a_i = n$. This inner maximum is achieved for $a_1
  = \ldots = a_{r-1} = 1$, $a_r = n-r+1$ and is equal to $\frac{1}{2}
  (n-r+1)(n-r)(r-1)$. Maximising this polynomial gives the result.

  \smallskip

  \textbf{(c)} Again, we can represent $S$ as a set of $m$ points and $P$ as a
  graph on the remaining $n-m$ points. This time, $P$ can only contain edges of
  the form $\{i,i-1\}$ for any $i$ such that neither $i$ nor $i-1$ are amongst the
  $m$ points of $S$. Hence $P$ is a disjoint union of paths with at most $n-m-1$
  edges, which is achieved if the points of $S$ are $1$ up to $m$ and $P$ is the
  path from $m+1$ to $n$. Together, we obtain a content of $m(n-m-1)$, maximised
  for $m = \lfloor (n-1)/2 \rfloor$ or $m = \lceil (n-1)/2 \rceil$.
\end{proof}

Table~\ref{t:opleagues} gives some values for the function $F^*(n,k)$ giving
the maximum content of a league where the parts of the partitions are 
intervals, together with the lower bound in (\ref{equation-3}).
Again, the zeros for $k=1$ and $k=n$ are omitted.

\begin{table}[ht]
$$\begin{array}{||r||r||r|r|r|r|r||}
\hline
n & \hbox{Total} & k=2 & 3 & 4 & 5 & 6 \\
\hline
2 & 0,0 &&&&&\\
3 & 1,1 & 1,1 &&&&\\
4 & 5,5 & 3,3 & 2,2 &&&\\
5 & 22,21 & 6,6 & 12,12 & 4,3 &&\\
6 & 88,84 & 12,10 & 40,40 & 30,30 & 6,4 &\\
7 & 345,330 & 20,15 & 100,100 & 150,150 & 66,60 & 9,5 \\
\hline
\end{array}$$
\caption{\label{t:opleagues}Values and bounds for $F^*(n,k)$ in the monoid of
order-preserving transformations.}
\end{table}

\section{The general linear semigroup}\label{section-general-linear}

For $q$ a prime power and $n$ a positive integer, let $\mathrm{GLS}(n,q)$
denote the semigroup of all linear maps on the $n$-dimensional vector space
$V$ over the Galois field $\GF(q)$ of order $q$. We have
$|\GLS(n,q)|=q^{n^2}$, since the linear maps are representable as $n\times n$
matrices. 

Our technique here resembles that in the case of the full transformation
semigroup. For $1\le k\le n$, the set of linear maps of rank at most $k$
forms an ideal, so we can analyse the principal factors. 

One important difference is that the structure is far more top-heavy. Indeed,
the group $\GL(n,q)$ of maps of full rank $n$ contains a non-zero proportion
of the whole semigroup.

\begin{prop}
Given $q$, there is a constant $c(q)$, with $0<c(q)<1$, so that
$$\lim_{n\to\infty}\frac{\mathrm{GL}(n,q)}{\mathrm{GLS}(n,q)}=c(q).$$ 
\label{p:ordergl}
\end{prop}

\begin{proof}
\begin{eqnarray*}
|\GL(n,q)| &=& \prod_{k=1}^n(q^n-q^{n-k}) \\
&=& q^{n^2}\prod_{k=1}^n (1-q^{-k}) \\
&\ge& |\GLS(n,q)|\prod_{k\ge1}(1-q^{-k}).
\end{eqnarray*}
The infinite product converges to a limit $c(q)>0$. Euler's Pentagonal
Numbers Theorem \cite[Theorem 4.1.3]{hall} gives
$$c(q)=\sum_{k\in\mathbb{Z}}(-1)^kq^{-k(3k-1)/2}=1-q^{-1}-q^{-2}+q^{-5}
+q^{-7}-q^{-12}-\cdots,$$
a form handy for calculation. For example, $c(2)=0.288788095\ldots$. In
fact, $c(q)$ is an evaluation of Jacobi's theta-function.
\end{proof}

The other main difference here is that the kernel of a linear map of rank $k$
is the partition of the vector space into cosets of a $(n-k)$-dimensional
subspace $U$ (the ``kernel'' of the map in the usual sense of linear algebra),
and a $k$-dimensional subspace $W$ is a transversal for the kernel partition
if and only if $U\cap W=\{0\}$. So the linear analogue of a league is a pair
$(P,S)$, where $P$ is a set of $(n-k)$-dimensional subspaces and $S$ a set
of $k$-dimensional subspaces such that, for all $U\in P$ and $W\in S$, we
have $U\cap W\ne\{0\}$. The simplest construction of a league is to take an
$(n-1)$-dimensional subspace $H$ of $V$, and to take $S$ and $P$ to consist of
all subspaces of the appropriate dimension contained in $H$; or dually, take
a $1$-dimensional subspace $K$ of $V$, and to take $S$ and $P$ to be all the
subspaces of the appropriate dimension containing $K$.

For $1\le k\le n$, the number of maps of rank $k$ is
$$\left(\gauss{n}{k}{q}\right)^2|\GL(k,q)|.$$
Here $$\gauss{n}{k}{q}$$ is the Gaussian coefficient, the number
of $k$-dimensional subspaces of an $n$-dimensional vector space over $\GF(q)$.
This coefficient is a monic polynomial in $q$ of degree $k(n-k)$ with
non-negative
integer coefficients, so is at least $q^{k(n-k)}$. Using the fact that
$|\GL(k,q)|\ge c(q)q^{k^2}$, we see that the number of maps of rank $k=n-d$
is at least $c(q)q^{n^2-d^2}$. So the largest principal factors are at the top.

The league just described in the principal factor of rank $k$ contains
$$\gauss{n-1}{k}{q}\gauss{n-1}{k-1}{q}$$ pairs. We have
\begin{eqnarray*}
\gauss{n-1}{k}{q}\gauss{n-1}{k-1}{q}\Bigg/\left(\gauss{n}{k}{q}\right)^2
&=& \frac{(q^k-1)(q^{n-k}-1)}{(q^n-1)^2} \\
&\ge& (1-1/q)^2q^{-n}.
\end{eqnarray*}
Altogether, we obtain a chain of length at least
\begin{eqnarray*}
        l(\GLS(n,q)) &\ge (1- 1/q)^2 q^{-n} \sum_{k=0}^{n-1} \left( {n \brack
        k}_q \right)^2 |\GL(k,q)| - 1\\
	&= (1- 1/q)^2 q^{-n}(|\GLS(n,q)| - |\GL(n,q)|) - 1.
\end{eqnarray*}
By Proposition~\ref{p:ordergl}, we have
$$
	|\GLS(n,q)| - |\GL(n,q)| \ge q^{n^2}(1-c(q)-o(1)),
$$
where the $o(1)$ is for fixed $q$ as $n\to\infty$. We obtain:

\begin{theorem}
$l(\GLS(n,q)) \ge (1-c(q)-o(1))(1- 1/q)^2 q^{-n} |\GLS(n,q)|$.
\end{theorem}

\section{Inverse semigroups}\label{section-inverse}

An \emph{inverse semigroup} is a semigroup $S$ such that for all $x\in S$, there
exists a unique $x^{-1}\in S$ where $xx^{-1}x=x$ and $x^{-1}xx^{-1} = x$.  The
\emph{symmetric inverse monoid} consists of the injective functions between
subsets of a fixed set $X$. It is the analogue of the symmetric group in the
context of inverse semigroups i.e.\ every inverse semigroup is isomorphic to an
inverse subsemigroup of some symmetric inverse monoid.

The length of the symmetric inverse monoid on any finite set was determined in
\cite{Ganyushkin2011aa}. However, the main theorem of \cite{Ganyushkin2011aa}
holds for arbitrary finite inverse semigroups, and the proof is essentially that
given in \cite{Ganyushkin2011aa}. We state the theorem in its full generality,
and give a slightly different proof from that in \cite{Ganyushkin2011aa}, which
makes use of the description of the maximal subsemigroups of a Rees matrix
semigroup given in \cite{Graham1968aa}.

Let $G$ be a group and let $n\in \mathbb{N}$. Then the \emph{Brandt semigroup}
$B(G, n)$ has elements $\left(\{1,\ldots, n\}\times G\times \{1,\ldots,
n\}\right)\cup \{0\}$ with multiplication defined by 
\begin{equation*}
  (i,g,j)(k,h,l)=
  \begin{cases}
    (i, gh, l)  & \text{if }j = k     \\
    0           & \text{if }j \not= k
  \end{cases}
\end{equation*}
and $0x=x0=0$ for all $x\in B(G, n)$. 

It follows from the Rees Theorem \cite[Theorem 5.1.8]{Howie1995aa} that the
principal factor of a $\J$-class $J$ of a finite inverse semigroup $S$ is
isomorphic to $B(G, n)$ where $G$ is any maximal subgroup of $S$ contained in
$J$ and $n$ is the number of $\mathscr{L}$- and $\mathscr{R}$-classes of $J$. 
Every inverse semigroup is regular, and so, to calculate the length of an
inverse semigroup, it suffices, by Lemma~\ref{lemma-regular}, to work out the
length of a Brandt semigroup.

\begin{prop}\label{prop-brandt}
  Let $G$ be a group and let $n\in \N$. Then:  
  \begin{equation}\label{formula}
    \begin{aligned}
      l(B(G,n)) & = n(l(G)+1)+\frac{n(n-1)}{2}|G|+n-1\\
                & = n(l(G)+2)+\frac{n(n-1)}{2}|G|-1.  
    \end{aligned}
  \end{equation}
  \label{thm-brandt}
\end{prop}

\begin{proof}
  We proceed by induction on $n$ and $|G|$. If $n=1$, then 
  $l(B(G, n))=l(G)+1$ and (\ref{formula}) holds. 
  
  Let $n\in \mathbb{N}$, $n>1$, and let $G$ be a finite group. Suppose that if
  either: ($m<n$ and $|H|=|G|$) or ($m=n$ and $|H|<|G|$), then 
  $$l(B(H, m))=m(l(H)+1)+\frac{m(m-1)}{2}|H|+m-1.$$
  We will show that (\ref{formula}) holds for $n$ and $G$.
 
  Remark 1 of \cite{Graham1968aa} implies that a maximal subsemigroup of $B(G,
  n)=(I\times G\times I)\cup \{0\}$ is isomorphic to either:
  \begin{enumerate}[(i)]
    \item $B(H, n)$ where $H$ is a maximal subgroup of $G$; or
    \item $B(G, n)\setminus (J\times G \times K)$ where $J$ and $K$ partition $I$. 
  \end{enumerate}
  (This is also shown directly in Theorem 6 of \cite{Ganyushkin2011aa}.)
  The semigroups of type (ii) are always maximal, while the ones in part (i)
  may or may not be.  It follows that either 
  $$l(B(G, n))=1+l(B(H, n))$$
  for some maximal subgroup $H$ of $G$, or 
  $$l(B(G, n))=1+l(B(G, n)\setminus (J\times G \times K))$$
  where $J$ and $K$ partition $I$. 

  In the latter case, 
  $$B(G, n)\setminus (J\times G \times K)=(J\times G\times J)\cup (K\times
  G\times K)\cup (K\times G\times J)\cup \{0\}.$$
  It is routine to verify that 
  $$(J\times G\times J)\cup\{0\}\cong B(G, |J|)\quad\hbox{ and }\quad (K\times
  G\times K)\cup \{0\}\cong B(G, |K|)$$
  and that 
  $$(K\times G\times J)\cup \{0\}$$  
  is a null ideal of $B(G, n)\setminus (J\times G \times K)$.  
  Thus, by applying Proposition~\ref{prop-ideal} to the null ideal and
  Lemma~\ref{lemma-regular} to the (regular!) quotient of $B(G, n)\setminus
  (J\times G \times K)$ by the ideal, it
  follows that 
  $$l(B(G, n)\setminus (J\times G \times K))=
  l(B(G, |J|))+l(B(G, |K|))+l(K\times G\times J\cup \{0\}).$$
  Since every non-empty subset of $(K\times G\times J)\cup \{0\}$ is a
  subsemigroup, it follows that
  $$l(K\times G\times J)\cup \{0\})=|J||K||G|$$
  and so by induction that 
  $$l(B(G, n)\setminus (J\times G \times K))=n(l(G)+1)+\frac{n(n-1)}{2}|G|+n-2.$$

  By the second part of the inductive hypothesis 
  \begin{eqnarray*}
    l(B(H, n))&=&n(l(H)+1)+\frac{n(n-1)}{2}|H|+n-1\\
    &\leq&n(l(G)+1)+\frac{n(n-1)}{2}|G|+n-2\\
    &=&l(B(G, n)\setminus (J\times G \times K)).
  \end{eqnarray*}
 Thus, when we are constructing a chain of semigroups, if we have a choice
between semigroups of types (i) or (ii), we should choose type (ii) to
obtain the longest possible chain. We conclude that
  $$l(B(G, n))=1+l(B(G, n)\setminus (J\times G \times K))$$
  and (\ref{formula}) holds. 
\end{proof}

The following result for inverse semigroups now follows immediately from
Lemma~\ref{lemma-regular} and Proposition~\ref{thm-brandt}.

\begin{theorem}[cf. Theorem 7 in \cite{Ganyushkin2011aa}]
  Let $S$ be a finite inverse semigroup with $\J$-classes $J_1, \ldots, J_m$.
  If $n_i\in \mathbb{N}$ denotes the number of $\mathscr{L}$- and
  $\mathscr{R}$-classes in $J_i$, and $G_i$ is any maximal subgroup of $S$
  contained in $J_i$, then
  \begin{eqnarray*}
    l(S) &=& -1+\sum_{i=1}^{m} l(B(G_i, n_i))\\
    &=& -1+\sum_{i=1}^m n_i(l(G_i)+1)+\frac{n_i(n_i-1)}{2}|G_i|+(n_i-1).
  \end{eqnarray*}
  \label{thm-inverse}
\end{theorem}

Given a specific inverse semigroup $S$, Theorem~\ref{thm-inverse} gives a
formula for $l(S)$ in terms of the numbers $n_i$ of $\L$- and $\R$-classes and
the lengths of the maximal subgroups $G$ of the $\J$-classes of $S$. Thus to
determine the length of a particular semigroup, it suffices to determine
these values. 

For example, if $I_n$ denotes the symmetric inverse monoid on an $n$-element set
and $x,y\in I_n$, then $x\J y$ if and only if the size of the domain of $x$ is
equal to the size of the domain of $y$; \cite[Exercise 5.11.2]{Howie1995aa}.
Hence the number of $\J$-classes in $I_n$ is $n+1$, corresponding to the
possible sizes of subsets of $\{1,\ldots, n\}$. If $J$ is the $\J$-class of
$I_n$, consisting of partial permutations defined on $i$ points, then the
number of $\L$- and $\R$-classes in $J$ is ${n\choose i}$ and every maximal
subgroup of $J$ is isomorphic to the symmetric group $S_i$ on $i$ points. So, in
the formula in Theorem~\ref{thm-inverse}, $m=n+1$, $n_i={n \choose i-1}$ and
$G_i=S_{i-1}$, so we have
$$l(I_n)=-1+\sum_{i=1}^{n+1}{n\choose i-1}(l(S_{i-1})+2)+
{n\choose i-1}\left({n\choose i-1}-1\right)\frac{(i-1)!}{2}-1,$$
where the values of $l(S_{i-1})$ for $i>1$ are given by
Theorem~\ref{symmetric_theorem} and $l(S_0)=0$.  The first few values of
$l(I_n)$ are given in Table~\ref{fig-inverse}, for further terms see
\cite{sloane1}. 

Three further examples are: the dual symmetric inverse monoid $I_n^*$ where
$m=n$, $n_i$ is the Stirling number of the second kind $S(n,i)$, and $G_i=S_i$;
see \cite[Theorem 2.2]{Fitzgerald1998aa}, Table~\ref{fig-inverse}, and
\cite{sloane2}; the partial injective order-preserving mappings
$POI_n$ on an $n$-element chain where $m=n+1$, $n_i={n\choose i-1}$, and $G_i$
is trivial; see \cite{Fernandes2001aa}, Table~\ref{fig-inverse}, and
\cite{sloane3}; the partial injective orientation-preserving
mappings $POPI_n$ on an $n$-element chain where $m=n+1$, $n_i={n\choose i-1}$,
and $G_i$ is the cyclic group with $i$ elements when $i>0$ and the trivial group
when $i=0$; see \cite{Fernandes2000aa}, Table~\ref{fig-inverse}, and
\cite{sloane4}.

\begin{table}[ht]
  \begin{center}
    \begin{tabular}{|c|c|c|c|c|c|c|c|c|c|}\hline
      $n$   &1&2&3&4&5&6&7&8&9\\\hline
      $l(I_n)$   &1& 6& 25& 116& 722& 5956& 59243& 667500& 8296060\\\hline
      $l(I_n^*)$ &0& 2& 17& 180& 3298& 88431& 3064050& 130905678&
      6732227475\\\hline
      $l(POI_n)$& 1& 5& 17& 53& 167& 550& 1899& 6809& 25067\\\hline
      $l(POPI_n)$&1& 6& 24& 92& 363& 1483& 6191& 26077& 109987\\\hline
    \end{tabular}
    \caption{The length of the longest chain of non-empty proper subsemigroups
    of some well-known inverse semigroups.} \label{fig-inverse}
  \end{center}
\end{table}

We consider the asymptotic value of $l(I_n)$ compared to $|I_n|$.

\begin{theorem}
  If $S$ is any of the symmetric inverse monoid $I_n$, the dual symmetric
  inverse monoid $I_n^*$, the partial order-preserving injective mappings
  $POI_n$, the partial orientation-preserving injective mappings $POPI_n$, then 
  $$\lim_{n\to\infty}\frac{l(S)}{|S|}=\frac{1}{2}.$$
\end{theorem}
\begin{proof}
  We present the proof in the case that $S=I_n$, the other proofs are
  similar.

  It is routine to check that 
  $$|I_n|=\sum_{i=0}^n {n\choose i}^2i!$$
  (see also \cite[Exercise 5.11.3]{Howie1995aa}). By Theorem~\ref{thm-inverse},
  \begin{eqnarray*}
    l(I_n)&=&-1+\sum_{i=0}^n \left[{n \choose i}(l(S_i)+1)+{n\choose i}\left({n
    \choose i}-1\right)\frac{i!}{2}+{n\choose i}-1\right]\\
    &=&\frac{|I_n|}{2}-1+\sum_{i=0}^n\left[ {n \choose i}(l(S_i)+1)-{n\choose
    i}\frac{i!}{2}+{n\choose i}-1\right]\\
    &=&\frac{|I_n|}{2}-n-2+\sum_{i=0}^n{n\choose
    i}\left[l(S_i)+2-\frac{i!}{2}\right]\\
    &=& \frac{|I_n|}{2} + \frac{n-1}{2}+\sum_{i=2}^n{n\choose
    i}\left[l(S_i)+2-\frac{i!}{2}\right].
  \end{eqnarray*}
  Note that, for $n\geq 1$,
  \begin{equation}\label{InEstimate}
    |I_n|\geq {n \choose n-1}^2(n-1)! = n \cdot n!
  \end{equation}
  and so to show that $l(I_n)$ is asymptotically $\frac{|I(n)|}{2}$ it suffices
  to show that the ratio of  
  $$\sum_{i=2}^n{n\choose i}\left[l(S_i)+2-\frac{i!}{2}\right]$$
  to $|I_n|$ tends to $0$ as $n\to\infty$. 
  By Theorem~\ref{symmetric_theorem}
  \begin{eqnarray*}
    \left| \sum_{i=2}^n{n\choose i}\left[l(S_i)+2-\frac{i!}{2}\right]\right| \leq
    \sum_{i=2}^n{n\choose i}\left[ \frac{3i}{2}+2+\frac{i!}{2}\right] 
    \leq \sum_{i=2}^n{n\choose i}i!
  \end{eqnarray*}
  Using the inequalities (\ref{InEstimate}) and
  \begin{equation*}
    |I_n|=\sum_{i=0}^n {n\choose i}^2i!\geq n\sum_{i=2}^{n-1} {n\choose i} i!
  \end{equation*}
  it follows that
  $$\frac{\sum_{i=2}^{n} {n\choose i} i!}{|I_n|}=
  \frac{n!}{|I_n|} + \frac{\sum_{i=2}^{n-1} {n\choose i} i!}{|I_n|}\leq
  \frac{2}{n} \rightarrow 0$$
  as $n\rightarrow \infty$ and the proof is complete.
\end{proof}

\subsection{Longest chains of inverse subsemigroups}

In this section we consider the question of determining the longest chains of
\emph{inverse} subsemigroups of a finite inverse semigroup. We define the
\textit{inverse subsemigroup length} of an inverse semigroup $S$ to be the
largest number of non-empty inverse subsemigroups of $S$ in a chain minus 1;
this is denoted $l^*(S)$. Since every group is an inverse semigroup, and every
subsemigroup of a finite group is a subgroup, if $G$ is a finite group, then
$l(G)=l^*(G)$. 

We will prove the following theorem. 

\begin{theorem}
  \label{thm-inverse-inverse}
  Let $S$ be a finite inverse semigroup with $\J$-classes $J_1, \ldots, J_m$.
  If $n_i\in \mathbb{N}$ denotes the number of $\mathscr{L}$- and
  $\mathscr{R}$-classes in $J_i$, and $G_i$ is any maximal subgroup of $S$
  contained in $J_i$, then
  \begin{eqnarray*}
    l^*(S) &=& -1+\sum_{i=1}^{m} l^*(B(G_i, n_i))\\
    &=& -1+\sum_{i=1}^m n_i(l(G_i)+1) + n_i - 1.
  \end{eqnarray*}
\end{theorem}

The proof is similar to the proof of Theorem~\ref{thm-inverse}.  We start by
proving analogues of Proposition~\ref{prop-ideal} and Lemma~\ref{lemma-regular}
for the inverse subsemigroup length, rather than length, of an inverse
semigroup.

To prove the analogue of Proposition~\ref{prop-ideal}, we require the following
facts about inverse semigroups. Let $S$ be an inverse semigroup, let $T$ and $U$
be inverse subsemigroups, and let $I$ be an ideal in $S$. Then the following are 
inverse semigroups: the ideal $I$, the quotient $S/I$, the intersection $T\cap
U$, and the union $T\cup I$. If $V$ is an inverse subsemigroup of $S/I$, then
$V\setminus\{0\} \cup I$ is an inverse subsemigroup of $S$. 

\begin{prop}
  \label{prop-ideal-inverse}
    Let $S$ be an inverse semigroup and let $I$ be an ideal of $S$. Then
    $l^*(S)=l^*(I)+l^*(S/I)$. 
\end{prop}

\begin{proof} 
  From the comments preceding the proposition, it is straightforward to verify
  that, the proof of this proposition follows by an argument analogous to that
  used to prove Proposition~\ref{prop-ideal}.
\end{proof}

The analogue of Lemma~\ref{lemma-regular}, follows as a corollary of
Proposition~\ref{prop-ideal-inverse} using the analogue of the proof of
Lemma~\ref{lemma-regular}.

\begin{cor} \label{cor-inverse}
  Let $S$ be a finite inverse semigroup and let $J_1, J_2, \ldots, J_m$ be the
  $\J$-classes of $S$. Then $l^*(S)=l^*(J_1^*)+l^*(J_2^*)+\cdots +l^*(J_m^*)-1$.
\end{cor}

As in the previous subsection, to calculate the inverse subsemigroup length of
an inverse semigroup, it suffices, by Corollary~\ref{cor-inverse}, to find the
inverse subsemigroup length of a Brandt semigroup.

\begin{prop}
  Let $G$ be a group and let $n\in \N$. Then:  
  \begin{equation}
    \label{inverse-formula}
    l ^ * (B(G,n)) = n (l(G) + 1) + n - 1 = n(l(G)+2)-1
  \end{equation}
\end{prop}

\begin{proof}
  We proceed by induction on $n$ and $|G|$. If $n=1$, then 
  $l^*(B(G, n))=l(G)+1$ and (\ref{inverse-formula}) holds. 
  
  Let $n\in \mathbb{N}$, $n > 1$, and let $G$ be a finite group. Suppose that if
  either: ($m < n$ and $|H| = |G|$) or ($m = n$ and $|H| < |G|$), then
  $$l^*(B(H, m))=m(l(H) + 1) + m - 1.$$
  We will show that (\ref{inverse-formula}) holds for $n$ and $G$.
 
  As in the proof of Proposition~\ref{prop-brandt}, a maximal subsemigroup of
  $B(G, n)=(I\times G\times I)\cup \{0\}$ is isomorphic to either:
  \begin{enumerate}[(i)]
    \item $B(H, n)$ where $H$ is a maximal subgroup of $G$; or
    \item $B(G, n)\setminus (J\times G \times K)$ where $J$ and $K$ partition $I$. 
  \end{enumerate}
  The subsemigroups of type (i) are inverse subsemigroups, and hence maximal
  inverse subsemigroups. The subsemigroups of type (ii) are not regular
  semigroups, since
  $$B(G, n)\setminus (J\times G \times K)=(J\times G\times J)\cup (K\times
  G\times K)\cup (K\times G\times J)\cup \{0\},$$
  and $(K\times G\times J)\cup \{0\}$ is a null subsemigroup. It follows that
  $$U := (J\times G\times J)\cup (K\times G\times K)\cup \{0\}$$
  is a maximal inverse subsemigroup of $B(G, n)\setminus (J\times G \times K)$,
  and hence of $B(G, n)$.  Since the $\mathscr{J}$-classes of $U$ are 
  $J\times G\times J$, $K\times G\times K$, and $\{0\}$, by 
  Corollary~\ref{cor-inverse}, 
  $$l^*(U)= l^*(B(G, |J|)) +  l^*(B(G, |K|)).$$
  Therefore either:
  $$l^*(B(G, n)) = 1 + l ^ * (B(H, n))$$
  for some maximal subgroup $H$ of $G$, or 
  $$l^*(B(G, n)) = 1 + l ^ * (B(G, m)) +  l ^ * (B(G, r))$$
  where $m + r = n$. By induction, and since $n > 1$, 
  \begin{eqnarray*}
    1 + l ^ * (B(G, m)) +  l ^ * (B(G, r)) & = & n (l(G) + 1) + n - 1 \\
    & > & n\ l(G) + n\\ 
    & = & n(l(H) + 1) + n\\
    & = & 1 + l ^ * (B(H, n)),
  \end{eqnarray*}
  and the result follows.
\end{proof}

In particular, we see that
$$l^*(I_n)=-1+\sum_{i=1}^{n+1}{n\choose i-1}(l(S_{i-1}+2)-1.$$
Some small values of the inverse subsemigroup lengths of the four examples of
inverse semigroups from the previous section can be seen in Table
\ref{fig-inverse-inverse}.

\begin{table}[ht]
  \begin{center}
    \begin{tabular}{|c|c|c|c|c|c|c|c|c|c|}\hline
      $n$   &1&2&3&4&5&6&7&8&9\\\hline
      $l^*(I_n)$  & 1 & 5 & 15 & 39 & 96 & 229 & 533 & 1217 & 2742\\\hline
      $l^*(I_n^*)$&0 & 2 & 11 & 49 & 223 & 1065 & 5337 & 28231 & 158939 \\\hline
      $l^*(POI_n)$& 1 & 4 & 11 & 26 & 57 & 120 & 247 & 502 & 1013\\\hline
      $l^*(POPI_n)$&1 & 6 & 17 & 44 & 97 & 208 & 429 & 884 & 1814 \\\hline
    \end{tabular}
    \caption{The length of the longest chain of non-empty proper inverse
    subsemigroups of some well-known inverse semigroups.}
    \label{fig-inverse-inverse}
  \end{center}
\end{table}

\section{Completely regular semigroups}\label{section-completely-regular}

In this section, we consider a special type of semigroup, which does not have
any leagues in any of its $\J$-classes.  A semigroup is
\emph{completely regular} if every element belongs to a subgroup. 

It follows by the Rees Theorem \cite[Theorems 3.2.3 and 4.1.3]{Howie1995aa} that
the principal factor of a $\J$-class $J$ of a finite completely regular
semigroup $S$ is isomorphic to a Rees 0-matrix semigroup $\mathcal{M}^0[I, G, J;
P]$ where $G$ is a finite group and $P$ is a $|J|\times |I|$ matrix with entries
in $G$. 

\begin{theorem}\label{thm-completely-regular}
  Let $S$ be a completely regular semigroup where the numbers of $\L$- and
  $\R$-classes are $m$ and $n$, and where the $\J$-classes of $S$ are $J_1,
  \ldots, J_r$. If $G_i$ is a maximal subgroup of $S$ contained in $J_i$, then 
  $$l(S)=m+n-r-1+\sum_{i=1}^{r} l(G_i).$$
\end{theorem}

\begin{proof}
  By Lemma~\ref{lemma-regular}, it suffices to show that 
  \begin{equation*}
    l(\mathcal{M}^0[I, G, J; P])=|I|+|J|+l(G)-1
  \end{equation*}
  where $\mathcal{M}^0[I, G, J; P]$ is a Rees 0-matrix semigroup over the group
  $G$ and $P$ is a $|J|\times |I|$ matrix with entries in $G$ (i.e.\ there are no
  entries equal $0$). Furthermore, since the length of a semigroup $S$ with zero
  adjoined is $1$ more than the length of $S$, it suffices to show that 
  \begin{equation*}
    l(\mathcal{M}[I, G, J; P])=|I|+|J|+l(G)-2.
  \end{equation*}
   where $R=\mathcal{M}[I, G, J; P]$ is a Rees matrix semigroup without zero.

  We proceed by induction on $|R|=|I|\times |G|\times |J|$. If
  $|I|=|J|=|G|=1$, then $|R|=1$ and so $l(R)=0$ and 
  $|I|+|J|+l(G)-2=1+1+0-2=0$. 

  As in the proof of Proposition~\ref{thm-brandt}, the length of $R$ is the
  length of one of its maximal subsemigroups plus $1$. Remark 1 of
  \cite{Graham1968aa} implies that a maximal subsemigroup of $R$ is isomorphic
  to one of:
  \begin{enumerate}[(i)]
    \item $I\setminus\{i\}\times G\times J$ for some $i\in I$;
    \item $I\times G\times J\setminus \{j\}$ for some $j\in J$;
    \item $\mathcal{M}[I, H, J; Q]$ where $H$ is a maximal subgroup of $G$ and
      $Q$ is a $|J|\times |I|$ matrix with entries in $H$. 
  \end{enumerate}
  Thus every maximal subsemigroup $T$ of $R$ is isomorphic to a completely regular
  Rees matrix semigroup. In any case, by induction, $l(T)=|I|+|J|+l(G)-3$, the
  result follows. 
\end{proof}

A semigroup $S$ is a \emph{band} if every element is an idempotent, i.e. $x^2=x$
for all $x\in S$. Every band is a completely regular semigroup where the maximal
subgroups are trivial, and so Theorem~\ref{thm-completely-regular} tells us that 
$$l(S)=m+n-r-1$$
where $m$, $n$, and $r$ are the numbers of $\L$-, $\R$-, and $\J$-classes of
$S$, respectively.

The $n$-generated \emph{free band} $B_n$ is the free object in the category of
bands, and, as it turns out, it is finite; see \cite[Section 4.5]{Howie1995aa}
for more details.  The $\J$-classes in $B_n$ are in 1-1 correspondence with the
non-empty subsets of $\{1,2,\ldots, n\}$, and the number of $\L$- and
$\R$-classes in any $\J$-class corresponding to a subset of size $k$ is:
$$k\prod_{i=1}^{k-2}(k-i)^{2^i}.$$
The following is an immediate corollary of these observations and
Theorem~\ref{thm-completely-regular}.

\begin{cor}
  The length of the free band $B_n$ with $n$ generators is:
  $$2\sum_{k=1}^{n}\left[{n\choose k}k\prod_{i=1}^{k-2}(k-i)^{2^i}\right]-2^n.$$
\end{cor}

Since every band with $n$ generators is a homomorphic image of the free band
$B_n$, it follows that $l(B_n)$ is an upper bound for $l(S)$ for every $n$
generated band $S$. 

\begin{table}[ht]\label{fig-free-band}
  \begin{center}
    \begin{tabular}{|c|c|c|c|c|c|c|}\hline
      $n$   &1&2&3&4&5&6\\\hline
      $l(B_n)$&0& 4& 34& 1264&
      3323778&33022614177128\\\hline
    \end{tabular}
    \caption{The length of the longest chain of non-empty proper subsemigroups
    in the free band $B_n$.}
  \end{center}
\end{table}

\section{Numbers of subsemigroups}\label{section-numbers}

Our technique for producing long chains also gives lower bounds for the number
of subsemigroups of certain semigroups.

We note that some results are known for groups. The number of subgroups of
$S_n$ is bounded below by roughly $2^{n^2/16}$. For this group contains an
elementary abelian subgroup of order $2^{\lfloor n/2\rfloor}$ generated by
$\lfloor n/2\rfloor$ disjoint transpositions; and an elementary abelian
group of order $2^m$ has 
$$\gauss{m}{k}{2}$$ 
subgroups of order $2^k$, this number being greater than $2^{k(m-k)}$, and so at
least $2^{\lfloor m^2/4\rfloor}$ when $k=\lfloor m/2\rfloor$. Remarkably,
Pyber~\cite{pyber} found an upper bound for the number of subgroups, also of the
form $2^{cn^2}$ for constant $c$.

If a null semigroup has $n$ non-zero elements, then it has $2^n$ subsemigroups,
since the zero together with any set of non-zero elements forms a subsemigroup.
So the existence of large null semigroups in principal factors of $T_n$, for
example, gives lower bounds for the number of subsemigroups, and on the number
of generators required.

\begin{theorem}
  Let 
  $$c = \frac{ \ee^{-2}}{3\sqrt{\ee^{-1} - 2 \ee^{-2}} \sqrt{3}}.$$
  Then
  \begin{enumerate}
    \item the number of subsemigroups of $T_n$ is at least $2^{(c - o(1))
      n^{n-1/2}}$;

    \item the smallest number $d(n)$ for which any subsemigroup of $T_n$ can be
      generated by $d(n)$ elements is at least $(c - o(1)) n^{n-1/2}$.
  \end{enumerate}
\end{theorem}

\begin{proof}
  The reader is reminded of the notation used in the proof of
  Theorem~\ref{thm-full-transformation-monoid}. We have exhibited then a null
  subsemigroup of $T_n$ of order $(n-k) N(n-1,k-1)$ for all $k \in \{1,\ldots,
  n\}$. We shall give a lower bound on the the order of the largest of those
  semigroups. In particular, we restrict ourselves to the set $J = \{k : |k -
  E(n-1)| < d \tau n^{1/2}\}$ where $E(n-1)$ is the expected rank of a
  transformation in $T_{n-1}$, $\tau = \sqrt{\ee^{-1} - 2 \ee^{-2}}$ and $d$ is a
  constant which we will specify later.  Using similar arguments as before, we
  can then prove that for all $k \in J$,
  \begin{eqnarray*}
    n - k &\ge& \ee^{-1} n - o(n)\\
    \sum_{k \in J} N(n-1,k-1) &\ge& \ee^{-1} \frac{d^2 - 1}{d^2} n^{n-1}\\
    \sum_{k \in J} (n-k) N(n-1,k-1) &\ge& \ee^{-1} \frac{d^2 - 1}{d^2} n^n -
    o(n^n),
  \end{eqnarray*}
  and hence the largest semigroup for $k \in J$  has order at least
  $$
  \frac{\ee^{-2}}{2\tau} \frac{d^2 - 1}{d^3} n^{n - 1/2} - o(n^{n - 1/2}).
  $$
  The fraction is maximised for $d = \sqrt{3}$. 
\end{proof}

\paragraph{Remark} Part (b) answers a question of Brendan McKay to the first
author a few years ago and gives a partial answer to Open Problem 1 in
\cite{Gray14}. The analogous number for $S_n$ (the smallest $d$ such that any
subgroup can be generated by at most $d$ elements) is only $\lfloor n/2\rfloor$
for $n>3$, as shown by McIver and Neumann~\cite{mn}.  Jerrum~\cite{jerrum} gave
a weaker bound $n-1$, but with a constructive (and computationally efficient)
proof.

\section{Open problems}

\paragraph{Problem 1} Does the ratio $l(T_n)/|T_n|$ tend to a limit as
$n\to\infty$? If so, what is this limit? Is it possible to improve on the
constant $\ee^{-2}$ by either more careful analysis, or counting the extra
steps available in a principal factor?

\paragraph{Problem 2} Evaluate the function $F(k,n)$ giving the largest
content of a league of rank $k$ on $\{1,\ldots,n\}$, and the function
$F^*(n,k)$ giving the largest content involving partitions into intervals.

\paragraph{Problem 3} In most cases, our results are not strong enough to
show that the number of subsemigroups of a semigroup $S$ is at least
$c^{|S|}$ for some $c>1$. Does such a result hold in the case $S=T_n$, for
example?

\paragraph{Problem 4} What can be said about the number of inverse
subsemigroups of an inverse semigroup, for example the symmetric inverse
semigroup $I_n$?

\end{document}